\documentclass{amsart}[12pt,article]

\usepackage{amsmath,amssymb,amscd,amsthm,indentfirst}
\usepackage{amsfonts,eufrak}
\usepackage[mathscr]{eucal}
\usepackage{cases}
\usepackage[all]{xy}
\xyoption{web}
\usepackage{enumerate}
\usepackage{hyperref}
\usepackage{lscape}
\usepackage{color}

\swapnumbers

\newtheorem{theorem}[equation]{Theorem}
\newtheorem{corollary}[equation]{Corollary}
\newtheorem{lemma}[equation]{Lemma}

\theoremstyle{definition}

\newtheorem{remark}[equation]{Remark}

\def\A{\ensuremath{\mathcal{A}}}

\def\C{\ensuremath{\mathcal{C}}}

\def\FF{\ensuremath{\mathbb{F}}}

\def\Out{\operatorname{Out}}

\newcommand{\GL}{\operatorname{GL}\nolimits}

\date{\today}
\title{The cohomology of the sporadic group $J_2$ over $\FF_3$}

\author{Antonio D\'{i}az Ramos}
\address{Departamento de {\'A}lgebra, Geometr{\'\i}a y Topolog{\'\i}a,
Universidad de M{\'a}\-la\-ga, Apdo correos 59, 29080 M{\'a}laga,
Spain.}
\email{adiaz@agt.cie.uma.es}

\author{Oihana Garaialde Oca\~{n}a}
\address{Departamento de Matem\'aticas, Facultad de Ciencia y Tecnolog\'ia,
Universidad del Pa\'is Vasco, Apartado 644, 48080 Bilbao, Spain.}
\email{ogaraialde@gmail.com}

\thanks{First author is supported by MICINN grant RYC-2010-05663. This work was partially supported by FEDER-MCI grant MTM2010-18089, Junta de Andaluc{\'\i}a grant FQM-213 and P07-FQM-2863.}

\begin{document}

\begin{abstract}
We describe the cohomology ring $H^*(J_2;\FF_3)$ both as subring of $H^*(3^{1+2}_+;\FF_3)$ and with an abstract presentation. We also give its Poincar\'{e} series. We use as tool a spectral sequence for the strongly closed $3$-subgroup of $J_2$. This method might be used to compute the cohomology of any finite simple group with a strongly closed $p$-subgroup.\\
MSC2010: 55T10, 55R35, 20J06, 20D20.\\
Keywords: cohomology ring, Hall-Janko group, spectral sequence.
\end{abstract}

\maketitle


\section{Introduction}

The computation of the cohomology of the sporadic finite simple groups is an intrinsically interesting subject in which only partial answers are yet known. By far, the most complete knowledge is at the prime $2$, see for instance \cite{AM2004} and \cite{BS2008}, where the authors tackle this problem using homotopy theory methods. Even with coefficients in $\FF_2$ not all the cohomology rings of the sporadic groups are fully understood. For $p$ odd and when the Sylow $p$-subgroup is the extraspecial group $p^{1+2}_+$ of order $p^3$ and exponent $p$, the description of the full cohomology ring has been circumvented, see \cite {TY1996} and \cite{Y1998}. In general, due to the complexity of the problem, computer calculations are in some cases the only or most complete source of information, see \cite{KINGGREENELLIS} for instance.

In this work, we present a method that involves the stable elements theorem \cite{CE1956} together with the spectral sequence \cite{D2012}. The approach is classical but the novelty is that this spectral sequence applies on the weak hypothesis of the existence of a strongly closed $p$-subgroup. Hence, this procedure can be utilized to compute the cohomology ring of any finite simple group which possesses a strongly closed $p$-subgroup. There is a classification of such finite simple groups \cite{FF2009}.

For the particular case of the Hall-Janko group or second Janko group $J_2$, the ring $H^*(J_2;\FF_2)$ was studied in \cite{CMM1999} with the partial aid of computer calculations. At the prime $3$, a Sylow $3$-subgroup $S$ of $J_2$ is isomorphic to $3^{1+2}_+$ and  $H^*(J_2;\FF_3)$ has been determined by computer \cite{KING}, \cite{KINGGREEN}. Using  our alternative method we give here  a computer-free description of the cohomology ring $H^*(J_2;\FF_3)$ both as a subring of $H^*(3^{1+2}_+;\FF_3)$ and with an abstract presentation.

Recall that in the finite simple group $J_2$ of order 604,800, the center $Z(S)\cong C_3$ is a strongly closed $3$-subgroup of $S$. Moreover, as the normalizer of $S$ controls fusion \cite[Remark 1.4]{RV2004}, the spectral sequence of \cite{D2012} amounts to the following: Firstly, set $E_*$ to be the Lyndon-Hochschild-Serre spectral sequence of the central extension
$$
C_3\rightarrow 3^{1+2}_+\rightarrow C_3\times C_3,
$$
with second page $E_2^{n,m}=H^n(C_3\times C_3;H^m(C_3;\FF_3))$ and converging to $H^*(3^{1+2}_+;\FF_3)$. Secondly, the group $\Out_{J_2}(S)=C_8$ acts on each page of this spectral sequence and taking invariants gives rise to a spectral sequence $E_*^{C_8}$ which converges to $H^*(J_2;\FF_3)$. In particular, the second page is $H^n(C_3\times C_3;H^m(C_3;\FF_3))^{C_8}$. The Lyndon-Hochschild-Serre spectral sequence $E_*$ of $3^{1+2}_+$ was computed by Leary in \cite{L1993} and it collapses in $E_6=E_\infty$. Hence, the invariants $E_6^{C_8}=E_\infty^{C_8}$ is a bigraded algebra associated to some filtration of $H^*(J_2;\FF_2)$ and we use this fact to determine the ring $H^*(J_2;\FF_3)=H^*(3^{1+2}_+;\FF_3)^{C_8}$.

The layout of the paper is as follows: We start in Section \ref{sectionLearyss} by providing a detailed description of the page $E_6$.  Then in Section \ref{sectioninvariants} we compute the invariants $E_6^{C_8}$ and the Poincar\'{e} series of $H^*(J_2;\FF_3)$. With the aid of the bigraded algebra $E_6^{C_8}$ we determine in Section \ref{sectionring} the ring $H^*(J_2;\FF_3)$ as a subring of $H^*(3^{1+2}_+;\FF_3)$ and an abstract presentation of $H^*(J_2;\FF_3)$.

\textbf{Acknowledgements:} We are grateful to Universidad de M\'alaga and Universidad del Pa\'is Vasco for their hospitality and travel and accommodation support in several occasions during the development of this work.

\section{The cohomology groups of $3^{1+2}_+$.}
\label{sectionLearyss}

We denote by $S=3^{1+2}_+$ the extraspecial group of order $27$ and exponent $3$. It has the following presentation
$$
S= \langle A,B,C | A^3=B^3=C^3=[A,C]=[B,C]=1\textit{, }[A,B]=C\rangle.
$$
The center of $S$ is $Z(S)=\langle C\rangle\cong C_3$ and hence we have the following central extension:
\begin{equation}\label{centralextension31+2+}
C_3 \rightarrow 3_+^{1+2} \stackrel{\pi}\rightarrow C_3\times C_3.
\end{equation}
Leary describes in \cite{L1993} the LHSss. $E_*$ of this extension. Its second page is given by
$$
E^{*,*}_2=H^*(C_3;\FF_3)\otimes H^*(C_3\times C_3;\FF_3)=\Lambda(u)\otimes \FF_3[t]\otimes \Lambda(y_1,y
_2)\otimes \FF_3[x_1,x_2],
$$
with the following degrees for the generators
$$
\deg(u)=\deg(y_1)=\deg(y_2)=1\textit{, }\deg(t)=\deg(x_1)=\deg(x_2)=2
$$ 
and with Bockstein operations $\beta(u)=t$ and $\beta(y_1)=x_1$, $\beta(y_2)=x_2$. The extension (\ref{centralextension31+2+}) is classified by $y_1y_2\in H^2(C_3\times C_3;\FF_3)$ and, according to \cite{L1993}, the differentials in $E_*$ are the following:
\begin{enumerate}[(i)]
\item $d_2(u)=y_1y_2$, $d_2(t)=0$,
\item $d_3(t)=x_1y_2-x_2y_1$,
\item $d_4(t^iu(x_1y_2-x_2y_1))=it^{i-1}(x_1x_2^2y_2-x_1^2x_2y_1)$, $d_4(t^2y_i)=u(x_1y_2-x_2y_1)x_i$,\item $d_5(t^2(x_1y_2-x_2y_1))=x_1^3x_2-x_1x_2^3$,  $d_5(ut^2y_1y_2)= k u(x_1^3y_2-x_2^3y_1)$, $k \ne 0$.
\end{enumerate}

A long and intricate computation leads from $E_2$ to $E_6$. We give an explicit description of $E_6$ but we omit most of the calculations.

\begin{lemma} \label{E6corner}With the notations above, the following table gives representatives of classes that form an $\FF_3$-basis of $E_6^{n,m}$ for $0\leq n\leq 6$ and $0\leq m\leq 5$:  
{\small
$$
\xymatrix@=0pt{
&&&&&&&\\
5& & & & & & &\\
4& & & & & & &\\
3& & &uty_1y_2 &  & & &\\
2& &ty_1, ty_2 & &ty_1x_1, ty_1x_2&  &tx_1^2y_1, tx_1^2y_2&\\
& && &ty_2x_2&  &tx_2^2y_1, tx_2^2y_2&\\
1& &uy_1, uy_2 &uy_1y_2 &uy_1x_1, uy_1x_2 & &ux_1^2y_1, ux_1^2y_2&\\
& && &uy_2x_1, uy_2x_2  & &ux_2^2y_1, ux_2^2y_2 &\\
0&1&y_1,y_2&x_1,x_2 &y_1x_1, y_1x_2&x_1^2, x_2^2&x_1^2y_1, x_1^2y_2 &x_1^3, x_2^3\\
&&&  & y_2x_2 &x_1x_2 &x_2^2y_1, x_2^2y_2 &x_1^2x_2, x_1x_2^2 \\
&0 & 1 & 2 & 3 & 4 & 5 & 6 
\ar@{-}"1,2"+<-5pt,-2pt>;"10,2"+<-5pt,-5pt>;
\ar@{-}"1,2"+<5pt,-2pt>;"10,2"+<5pt,-5pt>;
\ar@{-}"1,3"+<17pt,-2pt>;"10,3"+<17pt,-5pt>;
\ar@{-}"1,4"+<15pt,-2pt>;"10,4"+<15pt,-5pt>;
\ar@{-}"1,5"+<28pt,-2pt>;"10,5"+<28pt,-5pt>;
\ar@{-}"1,6"+<15pt,-2pt>;"10,6"+<15pt,-5pt>;
\ar@{-}"1,7"+<27pt,-2pt>;"10,7"+<27pt,-5pt>;
\ar@{-}"1,8"+<25pt,-2pt>;"10,8"+<25pt,-5pt>;
\ar@{-}"1,2"+<-5pt,-2pt>;"1,8"+<25pt,-2pt>;
\ar@{-}"2,2"+<-5pt,-5pt>;"2,8"+<25pt,-5pt>;
\ar@{-}"3,2"+<-5pt,-5pt>;"3,8"+<25pt,-5pt>;
\ar@{-}"4,2"+<-5pt,-5pt>;"4,8"+<25pt,-5pt>;
\ar@{-}"6,2"+<-5pt,-5pt>;"6,8"+<25pt,-5pt>;
\ar@{-}"8,2"+<-5pt,-5pt>;"8,8"+<25pt,-5pt>;
\ar@{-}"10,2"+<-5pt,-5pt>;"10,8"+<25pt,-5pt>;
}
$$
}
\end{lemma}

It turns out that this description of the corner of $E_6$ determines the rest of $E_6$ as there are both vertical and horizontal periodicities. More precisely, we show below that $E_6^{n,m}\cong E_6^{n,m+6}$ for $n,m\geq 0$ and that $E_6^{n,m}\cong E_6^{n+2,m}$ for $n\geq 5$ and $m\geq 0$. To recognize these isomorphisms we need nevertheless to fully understand the page $E_6$. The following result gives an explicit description of $E_6^{n,m}$ as a subquotient of $E_2^{n,m}$. 

\begin{lemma}\label{E6subquotient}
Assume the notations above and let $n\geq 0$ and $m\geq 0$. Set
$$
z=\begin{cases} t^{3\lfloor m/6\rfloor}, & m=0\mod 6\\ ut^{3\lfloor m/6\rfloor}, & m=1\mod 6\\
t^{3\lfloor m/6\rfloor+1},&m=2\mod 6\\
ut^{3\lfloor m/6\rfloor+1}, &m=3\mod 6.
\end{cases}
$$
Then we have:
\begin{enumerate}[(i)]
\item For $n=0$, $E_6^{n,m}=\langle z\rangle$ for $m=0\mod 6$ and $E_6^{n,m}=0$ otherwise.
\item For $n=1$, $E_6^{n,m}=\langle zy_1,zy_2\rangle$ for $m=0,1,2\mod 6$ and $E_6^{n,m}=0$ otherwise.
\item For $n=2$, $E_6^{n,m}=\langle zx_1,zx_2\rangle$ for $m=0\mod 6$, $E_6^{n,m}=\langle zy_1y_2\rangle$ for $m=1,3\mod 6$ and $E_6^{n,m}=0$ otherwise.
\item For $n=3$, $E_6^{n,m}=E_2^{n,m}$ for $m=1\mod 6$,  $E_6^{n,m}$ is the following quotient:
$$
\langle zx_1y_1,zx_1y_2,zx_2y_1,zx_2y_2\rangle
/\langle z(x_1y_2-x_2y_1)\rangle
$$
for $m=0,2\mod 6$ and $E_6^{n,m}=0$ otherwise.
\item For $n\geq 4$ and $n=2q$, $E_6^{n,m}$ is the following quotient if $m=0\mod 6$:
$$
\langle zx_1^ix_2^{q-i}\textit{, $0\leq i\leq q$}\rangle
/\langle zx_1^ix_2^{q-i}-zx_1^{i+2}x_2^{q-i-2}\textit{, $1\leq i\leq q-3$}\rangle,
$$
and $E_6^{n,m}=0$ otherwise.
\item For $n\geq 4$ and $n=2q+1$, if $m=0,1,2\mod 6$, then $E_6^{n,m}$ is the quotient of 
$$
\langle zx_1^ix_2^{q-i}y_1,zx_1^ix_2^{q-i}y_2\textit{, $0\leq i\leq q$}\rangle
$$ 
by 
{\small
$$
\langle zx_1^ix_2^{q-i}y_1-zx_1^{i+1}x_2^{q-i-1}y_2,
zx_1^jx_2^{q-j}y_2-zx_1^{j+1}x_2^{q-j-1}y_1
\textit{, $0\leq i\leq q-1$, $1\leq j\leq q-2$}\rangle,
$$
}
and $E_6^{n,m}=0$ otherwise.
\end{enumerate}
\end{lemma}

As a direct consequence of this lemma we obtain representatives for a basis of $E_6^{n,m}$ and any $n,m\geq 0$.

\begin{corollary}\label{E6representatives}
Assume the notations above and define $z$ as in Lemma \ref{E6subquotient}. For each $n\geq 0$ and $m\geq 0$ we denote by $B\leq E_2^{n,m}$ a set of elements such that their classes survive to $E_6^{n,m}$ and form a basis of $E_6^{n,m}$. Then we have:

\begin{enumerate}[(i)]
\item For $n=0$, $B=\{z\}$ for $m=0\mod 6$ and $B=\emptyset$ otherwise.
\item For $n=1$, $B=\{ zy_1,zy_2\}$ for $m=0,1,2\mod 6$ and $B=\emptyset$ otherwise.
\item For $n=2$, $B=\{zx_1,zx_2\}$ for $m=0\mod 6$, $B=\{zy_1y_2\}$ for $m=1,3\mod 6$ and $B=\emptyset$ otherwise.
\item For $n=3$, $B=\{zx_1y_1,zx_1y_2,zx_2y_1,zx_2y_2\}$ for $m=1\mod 6$, $B$ equals $\{zx_1y_1,zx_1y_2,zx_2y_2\}$ for $m=0,2\mod 6$ and $B=\emptyset$ otherwise.
\item For $n\geq 4$ and $n=2q$, $B=\{zx_1^q,zx_1^{q-1}x_2,zx_1^{q-2}x_2^2,zx_2^q\}$  if $m=0\mod 6$ and $B=\emptyset$ otherwise.
\item For $n\geq 4$ and $n=2q+1$, $B=\{zx_1^qy_1,zx_1^qy_2,zx_1^{q-1}x_2y_2,zx_2^qy_2\}$ if $m=0,1,2\mod 6$ and $B=\emptyset$ otherwise.
\end{enumerate}
\end{corollary}

\begin{lemma} \label{equverticaliso} With the notations above, the element $t^3$ survives to $E_6$. Moreover, for each $n\geq 0$ and $m\geq 0$, multiplication by $t^3$ is an isomorphism $\Psi^{n,m}\colon E_6^{n,m}\to E_6^{n,m+6}$ of vector spaces over $\FF_3$.
\end{lemma}
\begin{proof}
By Corollay \ref{E6representatives} or by an Evens' norm map argument \cite{E1991}, the element $t^3$ is a permanent cycle and hence it survives to $E_\infty=E_6$. In particular, multiplication by $t^3$ commutes with all the differentials and induces a linear map $\Psi^{n,m}_i\colon E^{n,m}_i\to E^{n,m+6}_i$ for all $n\geq 0$, $m\geq 0$ and $i\geq 2$.  From Corollary \ref{E6representatives} we deduce that $\Psi_6^{n,m}$ is an isomorphism.
\end{proof}

Now we can state the horizontal periodicity.

\begin{lemma} \label{equhorizontaliso} With the notations above, for each $n\geq 5$ and $m\geq 0$, the map $\Phi^{n,m}\colon E_6^{n,m}\to E_6^{n+2,m}$ given by 
$$
zx_1^q\mapsto zx_1^{q+1},zx_1^{q-1}x_2\mapsto zx_1^qx_2 ,zx_1^{q-2}x_2^2\mapsto zx_1^{q-1}x_2^2,zx_2^q\mapsto zx_2^{q+1}
$$
for $n=2q$ and $m=0\mod6$ and by
$$
zx_1^qy_1\mapsto zx_1^{q+1}y_1 ,zx_1^qy_2\mapsto zx_1^{q+1}y_2,zx_1^{q-1}x_2y_2\mapsto zx_1^qx_2y_2,zx_2^qy_2\mapsto zx_2^{q+1}y_2
$$
for $n=2q+1$ and $m=0,1,2\mod 6$ is an isomorphism of vector spaces over $\FF_3$.
\end{lemma}
\begin{proof}
This follows from Corollay \ref{E6representatives} by inspection.
\end{proof}

\section{Invariant elements in $E_6$}
\label{sectioninvariants}
In this section we compute $C_8$-invariants. Recall first that the outer automorphism group of $S\cong 3^{1+2}_+$ is $\GL_2(3)$, with the matrix $\big(\!\begin{smallmatrix} a & b
\\ c & d\end{smallmatrix}\!\big)$ acting as $A\mapsto A^aB^c$, $B\mapsto A^bB^d$ and $C\mapsto C^{ad-bc}$. As a generator of $C_8=\Out_{J_2}(S)$ we may choose $\big(\!\begin{smallmatrix} 1 & -1 \\ 1 & 1\end{smallmatrix}\!\big)$, which maps $A\mapsto AB$, $B\mapsto A^{-1}B$ and $C\mapsto C^{-1}$. The induced isomorphism in cohomology $E_2\rightarrow E_2$ is given by 
$$
y_1\mapsto y_1-y_2\textit{, }y_2\mapsto  y_1+y_2\textit{, }x_1\mapsto  x_1-x_2\textit{, }x_2\mapsto x_1+x_2
$$
and by $u\mapsto-u$, $t\mapsto-t$. We denote by $\alpha^{*,*}\colon E_6^{*,*}\to E_6^{*,*}$ the map induced in the sixth page. The following lemma shows that both $\Psi^2$ (Lemma \ref{equverticaliso}) and $\Phi^2$ (Lemma \ref{equhorizontaliso}) commute with $\alpha$. Hence, the computation of the $C_8$-invariants of $E_6$ is reduced to the computation of $C_8$-invariants in $\{E_6^{n,m}\}_{0\leq n\leq 8\textit{, }0\leq m\leq 11}$.

\begin{lemma}\label{commutes}
With the notations above we have:
\begin{enumerate}[(i)]
\item For each $n\geq 0$ and $m\geq 0$ the following square commutes:\label{commuteswithvertical}
$$
\xymatrix{
E_6^{n,m+12}\ar[rr]^{\alpha^{n,m+12}} && E_6^{n,m+12}\\
E_6^{n,m}\ar[u]^{(\Psi^{n,m})^2}\ar[rr]^{\alpha^{n,m}}&&E_6^{n,m}\ar[u]_{(\Psi^{n,m})^2}.
}
$$
\item For each $n\geq 5$ and $m\geq 0$ the following square commutes:\label{commuteswithhorizontal}
$$
\xymatrix{
E_6^{n,m}\ar[r]^{(\Phi^{n,m})^2} & E_6^{n+4,m}\\
E_6^{n,m}\ar[u]^{\alpha^{n,m}}\ar[r]^{(\Phi^{n,m})^2}&E_6^{n+4,m}\ar[u]_{\alpha^{n+4,m}}.
}
$$
\end{enumerate}
\end{lemma}
\begin{proof}
Part (\ref{commuteswithvertical}) is a direct consequence of $\alpha(t^6)=\alpha(t)^6=(-t)^6=t^6$. For part (\ref{commuteswithhorizontal}) the situation is more involved. By induction and using the relations in $(v)$ and $(vi)$ of Lemma \ref{E6subquotient} one can show that the following formula hold in $E_6^{2q,0}$ for $q\geq 2$:
$$
(x_1\pm x_2)^q=\begin{cases}
x_1^q+x_2^q\mp x_1^{q-1}x_2,&\textit{if $q$ is even}\\
x_1^q\pm x_2^q,&\textit{if $q$ is odd.}
\end{cases}
$$
From this, it is a direct computation to show that the equation
$(\Phi^{n,m})^2(\alpha^{n,m}(b))=\alpha^{n+4,m}((\Phi^{n,m})^2(b))$ holds for each generator $b$ of the basis $B$ described in points $(v)$ and $(vi)$ of Corollary \ref{E6representatives}. For instance, for $q$ even and the element $x_1^q$ 
we have:
$$\Phi^2(\alpha(x_1^q))=\Phi^2((x_1-x_2)^q)=\Phi^2(x_1^q+x_2^q+x_1^{q-1}x_2)=x_1^{q+2}+x_2^{q+2}+x_1^{q+1}x_2$$
and
$$\alpha(\Phi^2(x_1^q))=\alpha(x_1^{q+2})=(x_1-x_2)^{q+2}=x_1^{q+2}+x_2^{q+2}+x_1^{q+1}x_2.$$
%
\end{proof}

\begin{lemma} \label{E6invariantscorner}With the notations above, the following table gives representatives of classes that form an $\FF_3$-basis of $(E_6^{n,m})^{C_8}$ for $0\leq n\leq 8$ and $0\leq m\leq 11$:  
{\small
$$
\xymatrix@=0pt{
&&&&&&&&&\\
  11 & & & & & & & & & \\
  10 & & & & & & & & &\\ 
  9 & & &w_{2,9}& & & & & &\\ 
  8 & & & & & & & & t^4w_{7,0} &\\ 
  7 & & & & & & & & ut^3w_{7,0} &\\ 
  6 & & & &w_{3,6}&w_{4,6}& & & t^3w_{7,1}& w_{8,6}\\ 
  5 & & & & & & & & &\\ 
  4 & & & & & & & & &\\ 
  3 & & & & & & & & &\\ 
  2 & & & &w_{3,2}& & & & tw_{7,1}& \\ 
  1 & & &w_{2,1} &w_{3,1}& & & & uw_{7,1} &\\ 
    & & &        &w_{3,1}^- & & & & &\\
  0 &1 & & &  & & &&w_{7,0} &w_{8,0}\\
    &0 &1 &2 &3 &4 &5 &6 &7 &8
\ar@{-}"1,2"+<-5pt,-2pt>;"14,2"+<-5pt,-5pt>;
\ar@{-}"1,2"+<5pt,-2pt>;"14,2"+<5pt,-5pt>;
\ar@{-}"1,3"+<5pt,-2pt>;"14,3"+<5pt,-5pt>;
\ar@{-}"1,4"+<10pt,-2pt>;"14,4"+<10pt,-5pt>;
\ar@{-}"1,5"+<10pt,-2pt>;"14,5"+<10pt,-5pt>;
\ar@{-}"1,6"+<10pt,-2pt>;"14,6"+<10pt,-5pt>;
\ar@{-}"1,7"+<5pt,-2pt>;"14,7"+<5pt,-5pt>;
\ar@{-}"1,8"+<5pt,-2pt>;"14,8"+<5pt,-5pt>;
\ar@{-}"1,9"+<16pt,-2pt>;"14,9"+<16pt,-5pt>;
\ar@{-}"1,10"+<10pt,-2pt>;"14,10"+<10pt,-5pt>;
\ar@{-}"1,2"+<-5pt,-2pt>;"1,10"+<10pt,-2pt>;
\ar@{-}"2,2"+<-5pt,-5pt>;"2,10"+<10pt,-5pt>;
\ar@{-}"3,2"+<-5pt,-5pt>;"3,10"+<10pt,-5pt>;
\ar@{-}"4,2"+<-5pt,-5pt>;"4,10"+<10pt,-5pt>;
\ar@{-}"5,2"+<-5pt,-5pt>;"5,10"+<10pt,-5pt>;
\ar@{-}"6,2"+<-5pt,-5pt>;"6,10"+<10pt,-5pt>;
\ar@{-}"7,2"+<-5pt,-5pt>;"7,10"+<10pt,-5pt>;
\ar@{-}"8,2"+<-5pt,-5pt>;"8,10"+<10pt,-5pt>;
\ar@{-}"9,2"+<-5pt,-5pt>;"9,10"+<10pt,-5pt>;
\ar@{-}"10,2"+<-5pt,-5pt>;"10,10"+<10pt,-5pt>;
\ar@{-}"11,2"+<-5pt,-5pt>;"11,10"+<10pt,-5pt>;
\ar@{-}"13,2"+<-5pt,-7pt>;"13,10"+<10pt,-7pt>;
\ar@{-}"14,2"+<-5pt,-5pt>;"14,10"+<10pt,-5pt>;
}
$$
}
where 
$$
w_{2,1}=uy_1y_2, w_{3,1}=uy_1x_1+ uy_2x_2, w_{3,1}^-=uy_1x_2- uy_2x_1,
$$
$$
w_{3,2}=ty_1x_1+ ty_2x_2, w_{3,6}=t^3(y_1x_1+ y_2x_2), w_{4,6}=t^3(x_1^2+x_2^2),
$$
$$w_{2,9}=ut^4y_1y_2, w_{7,0}=x_1^3y_1+x_2^3y_2-x_1^2x_2y_2, w_{7,1}=x_1^3y_1+x_2^3y_2,$$ 
$$
w_{8,0}=x_1^4+x_2^4-x_1^2x_2^2, w_{8,6}=t^3x_1^4+t^3x_2^4.
$$
\end{lemma}
%
%

From this table and Lemma \ref{commutes}, the Poincar\'{e} series $P(t)$ of $J_2$ is as follows:
{\small
$$
\sum_{i=0}^\infty t^{12i}(1+t^3+2t^4+t^5+t^9+t^{10}+t^{11}+\sum_{j=0}^\infty t^{4j}(t^7+2t^8+t^9+t^{13}+2t^{14}+t^{15})),
$$
}
from where we obtain
\begin{equation}\label{equ_J2Ppoincareseries}
P(t)=\frac{1+t^3+t^4+t^5+t^9+t^{10}+t^{11}+t^{14}}{(1-t^4)(1-t^{12})}.
\end{equation}

\section{Ring structure}
\label{sectionring}

In this section we give a description of the ring $H^*(J_2;\FF_3)$ both as a subring of $H^*(3^{1+2}_+;\FF_3)$ and as an abstract ring via generators and relations. Recall that the ring $H^*(3^{1+2}_+;\FF_3)$ was described by Leary as follows:
\begin{theorem}[{\cite[Theorem 7]{L1992}}]\label{coho31+2+}
The ring $H^*(3^{1+2}_+;\FF_3)$ is generated by elements $y$, $y'$, $x$, $x'$, $Y$,
$Y'$, $X$, $X'$, $z$ with $$ \deg(y)=\deg(y')=1,
\deg(x)=\deg(x')=\deg(Y)=\deg(Y')=2,$$ $$ \deg(X)=\deg(X')=3
\mbox{ and } \deg(z)=6$$ subject to the following relations: $$
yy'=0,xy'=x'y,yY=y'Y'=xy',yY'=y'Y, $$ $$ YY'=xx',Y^2=xY',Y'^2=x'Y,
$$ $$ yX=xY-xx',y'X'=x'Y'-xx', $$ $$ Xy'=x'Y-xY', X'y=xY'-x'Y,$$
$$ XY=x'X,X'Y'=xX',XY'=-X'Y, xX'=-x'X, $$ $$ XX'=0,
x(xY'+x'Y)=-xx'^2, x'(xY'+x'Y)=-x'x^2, $$ $$ x^3y'-x'^3y=0,
x^3x'-x'^3x=0,$$ $$ x^3Y'+x'^3Y=-x^2x'^2 \mbox{ and }
x^3X'+x'^3X=0. $$ 
\end{theorem}
To study the ring of invariants $H^*(3^{1+2}_+;\FF_3)^{C_8}$ we need the action of the generator of $C_8$ on the generators described in the previous theorem. This action is
\begin{eqnarray}\label{actionofC_8onring}
y\mapsto y-y', y'\mapsto y+y',\\
x\mapsto x-x', x'\mapsto x+x',\nonumber\\
Y\mapsto x+x'-Y-Y', Y'\mapsto x-x'+Y-Y',\nonumber\\
X\mapsto -X-X',X'\mapsto X-X',\nonumber\\
z\mapsto -z.\nonumber
\end{eqnarray}
The action is determined by \cite[Theorem 7]{L1992} as follows: $y$ and $y'$ are the group homomorphisms dual to $A$ and $B$ respectively. The degree $2$ generators $Y$ and $Y'$ are the triple Massey products $Y=\langle y,y,y'\rangle$ and $Y'=\langle y',y',y\rangle$. Then $x$, $x'$, $X$ and $X'$ are the image by the Bockstein homomorphism of $y$, $y'$, $Y$ and $Y'$ respectively. To finish, the generator of $C_8$ maps $C$ to $C^{-1}$ and hence $z$ is mapped to $-z$.

Next, we determine the graded algebra $H^*(J_2;\FF_3)=H^*(3^{1+2}_+;\FF_3)^{C_8}$ using its associated bigraded algebra $E_6^{C_8}$. The link is provided as follows: the Lyndon-Hochschild-Serre spectral sequence $E_*$ converging to $H^*(3^{1+2}_+;\FF_3)$ comes equipped with a filtration $\{F^iH^n\}^{n+1}_{i=0}$ of $H^n=H^n(3^{1+2}_+;\FF_3)$  such that
$$
F^iH^n/F^{i+1}H^n\cong E_6^{i,n-i}\text{ for $i=0,\ldots,n$.}
$$
Now, as proven in \cite{D2012}, for the spectral sequence $E_*^{C_8}$ converging to $H^*(J_2;\FF_3)$, we have a filtration of $H^n(J_2;\FF_3)$ given by taking invariants in the previous filtration, $\{(F^iH^n)^{C_8}\}^{n+1}_{i=0}$, and this filtration satisfies
$$(F^iH^n)^{C_8}/(F^{i+1}H^n)^{C_8}\cong {E_6^{i,n-i}}^{C_8}\text{ for $i=0,\ldots,n$.}
$$

For a class $c\in F^iH^n\setminus F^{i+1}H^n$ set $\overline c\in F^iH^n/F^{i+1}H^n$ to be the non-zero image of $c$ in this quotient. From the previous discussion, if $c\in H^*(J_2;\FF_3)$ then $\overline{c}$ belongs to $E_6^{C_8}$. For the generators in Theorem \ref{coho31+2+} we have the following, where the last line is consequence of \cite[Lemma 2.4, Lemma 2.13]{L1992thesis}: \begin{eqnarray}\label{equ_overlinegenerators}
\overline{y}=y_1, \overline{y'}=y_2, \overline{x}=x_1, \overline{x'}=x_2,\\
\overline Y=uy_1, \overline{Y'}=-uy_2, \overline{X}=ty_1, \overline{X'}=-ty_2, \overline{z}=t^3.\nonumber
\end{eqnarray}

To find generators for $H^*(J_2;\FF_3)$ it is enough to find generators of $E_6^{C_8}$ and then lift them. A set of generators for $E_6^{C_8}$ is given by (see Table \ref{table_E6_21x21} and Lemma \ref{E6invariantscorner}):
$$
w_{2,1},w_{3,1},w^{-}_{3,1},w_{3,2},w_{3,6},w_{4,6},w_{2,9},t^6,
$$
$$
w_{7,0},w_{8,0},uw_{7,1},tw_{7,1},t^3w_{7,1},ut^3w_{7,0},t^4w_{7,0},w_{8,6},w_{11,0},w_{12,0},
$$
where $w_{11,0}=\Phi^2(w_{7,0})=x_1^5y_1+x_2^5y_2-x_1^4x_2y_2$, $w_{12,0}=\Phi^2(w_{8,0})=x_1^6+x_2^6-x_1^4x_2^2$. Nevertheless, as we show in the next result, it is enough to lift the first $8$ generators listed above. To this purpose, note that the Poincar\'{e} series of $H^*(J_2;\FF_3)$ is (\ref{equ_J2Ppoincareseries})
$$
P(t)=1+t^3+2t^4+t^5+t^7+2t^8+2t^9+t^{10}+2t^{11}+3t^{12}+\ldots.
$$
From this expression we deduce that the lifts of the generators $w_{2,1},w_{3,1},w^{-}_{3,1},w_{3,2}$ are the linear generators of  $H^n(3^{1+2}_+;\FF_3)^{C_8}$ for $n=3,4,5$. For the generators $w_{3,6},w_{4,6},w_{2,9}$, as they are multiples of $t^3$, an appropriate guess is that their lifts are of the form $zc$, where $c$ is mapped to $-c$ under the action and $c$ has total degree $3$, $4$ or $5$. To sum up, we need to solve some linear algebra problems in $H^n(3^{1+2}_+;\FF_3)$ with $n\in \{3,4,5\}$. Basis for these vector spaces are described in \cite[Lemma 2.4(2), proof of Theorem 2.15]{L1992thesis}.


\begin{theorem}\label{thmcohoJ2subring}
The ring $H^*(J_2;\FF_3)$ is the subring of $H^*(3^{1+2}_+;\FF_3)$ generated by elements $a,b,c,d,e,f,g,h$ with
$$a=Yy'-xy-x'y',b=Yx-Y'x',c=x^2+x'^2+xY'+x'Y$$
$$d=Xx-X'x',e=z(yx+y'x'),f=z(x^2+x'^2),$$
$$g=-z(Xx-X'x'+YX'),h=z^2$$
and degrees 
$$
\deg(a)=3, \deg(b)=\deg(c)=4, \deg(d)=5,
$$ 
$$
\deg(e)=9, \deg(f)=10, \deg(g)=11, \deg(h)=12.
$$
\end{theorem}

\begin{proof}
It is a straightforward computation using the action of $C_8$ (Equation \ref{actionofC_8onring}) and the presentation of $H^*(3^{1+2}_+;\FF_3)$ (Theorem \ref{coho31+2+}) that the elements in the statement are indeed invariant under this action. For instance, for $g=-z(xX-x'X'+YX')$, its image under the action is 
\begin{align*}
z[(x-x')(-X-X')-(x+x')(X-X')+(x+x'-Y-Y')(X-X')]\\
=z(-xX-xX'+x'X+x'X' -xX+xX'-x'X+x'X'+xX-xX'\\
+x'X-x'X'-YX+YX'-Y'X+Y'X' )\\
=z(-xX-xX'+x'X+x'X'-YX+YX'-Y'X+Y'X' )\\
=z(-xX-xX'+x'X+x'X'-x'X-YX'+xX' )\\
=z(-xX+x'X'-YX')=-z(xX-x'X'+YX').
\end{align*}
Now, using Equation \ref{equ_overlinegenerators} we obtain:
$$
\overline{a}=uy_1y_2,\overline{b}=uy_1x_1+uy_2x_2, \overline{c}=uy_1x_2-uy_2x_1,
$$
$$\overline{d}=t(y_1x_1+y_2x_2),\overline{e}=t^3(y_1x_1+y_2x_2),\overline{f}=t^3(x_1^2+x_2^2),$$
$$
\overline{g}=ut^4y_1y_2, \overline{h}=t^6,
$$
or, with the notations of Lemma \ref{E6invariantscorner},
$$
\overline{a}=w_{2,1},\overline{b}=w_{3,1}, \overline{c}=w^{-}_{3,1},
\overline{d}=w_{3,2},\overline{e}=w_{3,6},\overline{f}=w_{4,6},
\overline{g}=w_{2,9},\overline{h}=t^6.
$$
To prove that $a,b,c,d,e,f,g,h$ generate $H^*(J_2;\FF_3)$ we show that the  subalgebra $\langle a,b,c,d,e,f,g\rangle$ generates the full bigraded algebra $E_6^{C_8}$. As $\overline{h}=t^6$, multiplication in $E_6^{C_8}$ by $\overline{h}$ gives exactly the vertical periodicity $\Psi^2$ of $(a)$ in Lemma \ref{commutes}. So we just need to show that the rows from $0$ to $11$ are generated by the overlines of elements in $\langle a,b,c,d,e,f,g\rangle$. The rest of the elements in the table of Lemma \ref{E6invariantscorner} are generated as follows:
$$
\overline{-ac}=w_{7,0},\overline{c^2}=w_{8,0},\overline{bc}=uw_{7,1},\overline{dc}=tw_{7,1},$$
$$\overline{ec}=t^3w_{7,1},\overline{fc}=w_{8,6},
\overline{ga}=ut^3w_{7,0},\overline{-gc}=t^4w_{7,0}.
$$
Now, the following computation shows that the horizontal periodicity $\Phi^2$ of $(b)$ of Lemma \ref{commutes} is realized by multiplication by $c$ at the ring level:
\begin{align*}
\overline{-ac^i}=\overline{x^{2i+1}y+x'^{2i+1}y'-x^{2i}x'y'}=&\Phi^{2(i-1)}(w_{7,0})\in {E_6^{7+4(i-1),0}}^{C_8}\\
\overline{bc^i}=\overline{Yx^{2i+1}-Y'x'^{2i+1}}=&\Phi^{2(i-1)}(uw_{7,1})\in {E_6^{7+4(i-1),1}}^{C_8}\\
\overline{dc^i}=\overline{x^{2i+1}X-x'^{2i+1}X'}=&\Phi^{2(i-1)}(tw_{7,1})\in{E_6^{7+4(i-1),2}}^{C_8}\\
\overline{ec^i}=\overline{z(yx^{2i+1}+y'x'^{2i+1})}=&\Phi^{2(i-1)}(t^3w_{7,1})\in{E_6^{7+4(i-1),6}}^{C_8}\\
\overline{fc^i}=\overline{z(x^{2(i+1)}+x'^{2(i+1)})}=&\Phi^{2(i-1)}(w_{8,6})\in{E_6^{8+4(i-1),6}}^{C_8}\\
\overline{gac^i}=\overline{z(x^{2i+3}Y-x'^{2i+3}Y'+x^{2(i+1)}x'Y')}=&\Phi^{2i}(ut^3w_{7,0})\in{E_6^{7+4i,7}}^{C_8}\\
\overline{-gc^i}=\overline{z(x^{2i+1}X-x'^{2i+1}X'+x^{2i}x'X')} =& \Phi^{2(i-1)}(t^4w_{7,0}) \in{E_6^{7+4(i-1),8}}^{C_8}
\end{align*}
for $i\geq 1$ and the following for $i\geq 2$:
\begin{align*}
\overline{c^i}=\overline{x^{2i}+x'^{2i}-x^{2(i-1)}x'^2}=&\Phi^{2(i-2)}(w_{8,0})\in{E_6^{8+4(i-2),0}}^{C_8}.
\end{align*}
This finishes the proof.
\end{proof}

\begin{remark}
The action of the Steenrod algebra $\A_3$ on $H^*(J_2;\FF_3)$ can be easily deduced from its action on $H^*(3^{1+2}_+;\FF_3)$, which was described by Leary in \cite{L1992}.
\end{remark}

The previous theorem shows that the corner of the bigraded algebra $E_6^{C_8}$ is generated by the following overlined elements (cf. Lemma \ref{E6invariantscorner}): 
$$
\xymatrix@=0pt{
&&&&&&&&&&&&&\\
  11 & & & & & & & & & & & & &\\
  10 & & & & & & & & & & & & &\\ 
  9 & & &\overline{g} & & & & & & & & & &\\ 
  8 & & & & & & & & \overline{-gc} & & & &\overline{-gc^2} &\\ 
  7 & & & & & & & & \overline{ga} & & & &\overline{gac} &\\ 
  6 & & & &\overline{e} &\overline{f} & & & \overline{ec}& \overline{fc} & & &\overline{ec^2} &\overline{fc^2}\\ 
  5 & & & & & & & & & & & & &\\ 
  4 & & & & & & & & & & & & &\\ 
  3 & & & & & & & & & & & & &\\ 
  2 & & & &\overline{d}& & & & \overline{dc}& & & & \overline{dc^2}&\\ 
  1 & & &\overline{a} &\overline{b}, \overline{c} & & & & \overline{bc} & & & &\overline{bc^2} &\\ 
  0 &1 & & &  & & &&\overline{-ac} &\overline{c^2} & & &\overline{-ac^2} &\overline{c^3}\\
    &0 &1 &2 &3 &4 &5 &6 &7 &8 &9 &10 &11 &12
\ar@{-}"1,2"+<-5pt,-2pt>;"14,2"+<-5pt,-5pt>;
\ar@{-}"1,2"+<5pt,-2pt>;"14,2"+<5pt,-5pt>;
\ar@{-}"1,3"+<6pt,-2pt>;"14,3"+<6pt,-5pt>;
\ar@{-}"1,4"+<7pt,-2pt>;"14,4"+<7pt,-5pt>;
\ar@{-}"1,5"+<10pt,-2pt>;"14,5"+<10pt,-5pt>;
\ar@{-}"1,6"+<7pt,-2pt>;"14,6"+<7pt,-5pt>;
\ar@{-}"1,7"+<6pt,-2pt>;"14,7"+<6pt,-5pt>;
\ar@{-}"1,8"+<5pt,-2pt>;"14,8"+<5pt,-5pt>;
\ar@{-}"1,9"+<13pt,-2pt>;"14,9"+<13pt,-5pt>;
\ar@{-}"1,10"+<8pt,-2pt>;"14,10"+<8pt,-5pt>;
\ar@{-}"1,11"+<6pt,-2pt>;"14,11"+<6pt,-5pt>;
\ar@{-}"1,12"+<7pt,-2pt>;"14,12"+<7pt,-5pt>;
\ar@{-}"1,13"+<12pt,-2pt>;"14,13"+<12pt,-5pt>;
\ar@{-}"1,14"+<10pt,-2pt>;"14,14"+<10pt,-5pt>;
\ar@{-}"1,2"+<-5pt,-2pt>;"1,10"+<83pt,-2pt>;
\ar@{-}"2,2"+<-5pt,-5pt>;"2,10"+<83pt,-5pt>;
\ar@{-}"3,2"+<-5pt,-5pt>;"3,10"+<83pt,-5pt>;
\ar@{-}"4,2"+<-5pt,-5pt>;"4,10"+<83pt,-5pt>;
\ar@{-}"5,2"+<-5pt,-5pt>;"5,10"+<83pt,-5pt>;
\ar@{-}"6,2"+<-5pt,-5pt>;"6,10"+<83pt,-5pt>;
\ar@{-}"7,2"+<-5pt,-5pt>;"7,10"+<83pt,-5pt>;
\ar@{-}"8,2"+<-5pt,-5pt>;"8,10"+<83pt,-5pt>;
\ar@{-}"9,2"+<-5pt,-5pt>;"9,10"+<83pt,-5pt>;
\ar@{-}"10,2"+<-5pt,-5pt>;"10,10"+<83pt,-5pt>;
\ar@{-}"11,2"+<-5pt,-5pt>;"11,10"+<83pt,-5pt>;
\ar@{-}"12,2"+<-5pt,-5pt>;"12,10"+<83pt,-5pt>;
\ar@{-}"13,2"+<-5pt,-5pt>;"13,10"+<83pt,-5pt>;
\ar@{-}"14,2"+<-5pt,-5pt>;"14,10"+<83pt,-5pt>;
}
$$
By inspection, columns $11$ and $12$ are obtained from columns $7$ and $8$ via multiplication by $c$ at the ring level. In fact, we proved in Theorem \ref{thmcohoJ2subring} that the rest of the columns are obtained by multiplying by the succesive powers of $c$, and the missing rows by multiplying  by powers of $h$.

\begin{theorem}\label{thmcohoJ2abstractpresentation}
The ring $H^*(J_2;\FF_3)$ is the free graded-commutative algebra on the generators $a,b,c,d,e,f,g,h$ with 
$$
\deg(a)=3, \deg(b)=\deg(c)=4, \deg(d)=5,
$$ 
$$
\deg(e)=9, \deg(f)=10, \deg(g)=11, \deg(h)=12
$$
subject to the following relations:
$$
ab=0,b^2=0,bc+ad=0,bd=0,ae=0,be=0,fa+ce=0,bf=ag,ed=ag,
$$
$$
bg=0,fd+cg=0,dg=0,ef+ach=0,eg=had,f^2=c^2h,fg+chd=0.
$$
\end{theorem}
\begin{proof}
By Theorem \ref{thmcohoJ2subring} we know that there is a surjective homomorphism from the free graded-commutative algebra $R$ on generators $a,b,c,d,e,f,g,h$ of the stated degrees to $H^*(J_2;\FF_3)$. Each generator of $R$ is mapped to the generator with the same name in $H^*(J_2;\FF_3)$. To compute the kernel of this homomorphism we start looking at relations in $E_6^{C_8}$ among the elements $\overline{a},\overline{b},\overline{c},\overline{d},\overline{e},\overline{f},\overline{g},\overline{h}$. Then we lift these relations back to $H^*(J_2;\FF_3)$. To decide whether we have considered enough relations we use Poincar\'{e} series. We find that the highest order relation occurs in total degree $21$.

Recall from the proof of Theorem \ref{thmcohoJ2subring} that 
$$
\overline{a}=w_{2,1},\overline{b}=w_{3,1}, \overline{c}=w^{-}_{3,1},
\overline{d}=w_{3,2},\overline{e}=w_{3,6},\overline{f}=w_{4,6},
\overline{g}=w_{2,9}.
$$
Now, from Table \ref{table_E6_21x21}, where we show the corner $\{{E_6^{n,m}}^{C_8}\}_{n,m=0,\ldots,21}$, we find that the following $16$ products are $0$:
$$
\overline{a}\overline{b},\overline{b}^2,\overline{a}\overline{d},\overline{b}\overline{d},\overline{a}\overline{e},\overline{b}\overline{e},\overline{a}\overline{f},\overline{a}\overline{g}, \overline{e}\overline{d},\overline{b}\overline{g},\overline{c}\overline{g},\overline{d}\overline{g},\overline{h}\overline{c}\overline{a},\overline{e}\overline{g},\overline{h}\overline{c}^2,\overline{f}\overline{g}.
$$
All these products are zero since they have bidegree $(n,m)$ with ${E_6^{n,m}}^{C_8}=0$. Now assume some product $\overline{a}^{i_a}\overline{b}^{i_b}\overline{c}^{i_c}\overline{d}^{i_d}\overline{e}^{i_e}\overline{f}^{i_f}\overline{g}^{i_g}\overline{h}^{i_h}$ is $0$ in ${E_6^{i,n-i}}^{C_8}$, which is isomorphic to $(F^iH^n)^{C_8}/(F^{i+1}H^n)^{C_8}$. Here we have $n=3i_a+4i_b+4i_c+5i_d+9i_e+10i_f+11i_g$ and $i=2i_a+3i_b+3i_c+3i_d+3i_3+4i_f+2i_g$. Then the element ${a}^{i_a}{b}^{i_b}{c}^{i_c}{d}^{i_d}{e}^{i_e}{f}^{i_f}{g}^{i_g}{h}^{i_h}$ of $(F^iH^n)^{C_8}$ lies in $(F^{i+1}H^n)^{C_8}$ too and hence it may be written using elements from the latter. This gives rise to a relation. For instance, for $\overline{a}\overline{d}=0\in {E_6^{5,3}}^{C_8}$, and using the fact that 
$$
{E_6^{6,2}}^{C_8}=0, {E_6^{7,1}}^{C_8}=\langle \overline{bc}\rangle,{E_6^{8,0}}^{C_8}=\langle \overline{c^2}\rangle,
$$
we deduce that $ad=\alpha bc+\beta c^2$ holds in $H^*(J_2;\FF_3)$ for some $\alpha$ and $\beta$ in $\FF_3$. In this case, the equation is satisfied for $\alpha=-1$ and $\beta=0$:
\begin{align*}
bc=(Yx-Y'x')(x^2+x'^2+xY'+x'Y)\\
=Yx^3+Yxx'^2+YY'x^2+Y^2xx'-x^2x'Y'-x'^3Y'-xx'Y'^2-x'^2YY'\\
=Yx^3+Yxx'^2+x^3x'+x^2x'Y'-x^2x'Y'-x'^3Y'-xx'^2Y-xx'^3\\
=Yx^3-Y'x'^3
\end{align*}
and
\begin{align*}
ad=(Yy'-xy-x'y')(Xx-X'x')\\
=xy'YX-x'y'YX'-x^2yX+xx'yX'-xx'y'X+x'^2y'X'\\
=xx'y'X-x^2y'X'-x^2yX+x^2y'X'-xx'y'X+x'^2y'X'\\
=-x^2yX+x'^2y'X'=-x^2(xY-xx')+x'^2(x'Y'-xx')\\
=-x^3Y+x^3x'+x'^3Y'-xx'^3=-Yx^3+Y'x'^3.
\end{align*}
The same analysis for the products $\overline{a}\overline{f},\overline{a}\overline{g}, \overline{e}\overline{d},\overline{c}\overline{g},\overline{h}\overline{c}\overline{a},\overline{e}\overline{g},\overline{h}\overline{c}^2,\overline{f}\overline{g}
$ produce the following relations respectively:
\begin{align*}
af=-ec=-z(x^3y+x'^3y'),\\
ag=bf=z(x^3Y-x'^3Y'-xx'^2Y+x^3x'),\\
ed=ag=z(x^3Y-x'^3Y'-xx'^2Y+x^3x'), \\
cg=-fd=-z(Xx^3-X'x'^3+X'x^2x'),\\
ach=-ef=-z^2(x^3y+x'^3y'-x^2x'y'),\\
eg=had=z^2(-x^3Y+x'^3Y'),\\
c^2h=f^2=z^2(x^4+x'^4-x^2x'^2),\\
fg=-chd=z^2(-Xx^3+X'x'^3).
\end{align*}
For the remaining products $\overline{a}\overline{b},\overline{b}^2,\overline{b}\overline{d},\overline{a}\overline{e},\overline{b}\overline{e},\overline{b}\overline{g},\overline{d}\overline{g}$ we obtain $ab=b^2=bd=ae=be=bg=dg=0$. For example,
\begin{align*}
ab=(Yy'-xy-x'y')(Yx-Y'x')\\
=Y^2xy'-YY'x'y'-Yx^2y+Y'xx'y-Yxx'y'+Y'x'^2y'\\
=Y'x^2y'-xx'^2y'-Yx^2y+Y'x^2y'-Yx'^2y+Y'x'^2y'\\
=x^3y'-x'^3y-x^3y'+x^3y'-xx'^2y'+xx'^2y'=x^3y'-x'^3y=0.
\end{align*}
Therefore, we get the following relations:
$$
ab, b^2, bc+ad,bd,ae,be,fa+ce,bf-ag,ed-ag,
$$
$$
bg,fd+cg,dg,ef+cha,eg-had,f^2-c^2h,fg+chd.
$$
Set $I$ to be the ideal of $R$ generated by these relations and set $H=R/I$. We show that $H\cong H^*(J_2;\FF_3)$. The subalgebra of $H$ generated by $c$ and $h$ is the polynomial algebra on a generator of degree $4$ and a generator of degree $12$, $\FF_3[c,h]$. Using that $H$ is graded-commutative and the $16$ relations above it is easy to check that every monomial $m={a}^{i_a}{b}^{i_b}{c}^{i_c}{d}^{i_d}{e}^{i_e}{f}^{i_f}{g}^{i_g}{h}^{i_h}$ of $H$ may be rewritten as $m=m'c^{i'_c}h^{i'_h}$ with $m'\in \{1,a,b,d,e,f,g,ga\}$. Hence,  $H$ is finite over $\FF_3[c,h]$,
\begin{equation}\label{HJ2finitegeneratorsoverch}
H=\FF_3[c,h]\{1,a,b,d,e,f,g,ga\},
\end{equation}
and its Poincar\'{e} series coincides with that of $J_2$ (\ref{equ_J2Ppoincareseries}). It follows that $H\cong H^*(J_2;\FF_3)$.

\end{proof}
\begin{remark}
As noted in the introduction, a presentation of $H^*(J_2;\FF_3)$ was already obtained by computer \cite{KING}. There the generators are denoted by 
$$
a_{3,0}, a_{4,1}, b_{4,0}, a_{5,0}, a_{9,1}, b_{10,0}, a_{11,1}, c_{12,1},
$$
with degrees given by the first coordinate of the sub-index. An isomorphism with the presentation of Theorem \ref{thmcohoJ2abstractpresentation} is given by 
$$
a \mapsto -a_{3,0}, b \mapsto a_{4,1}, c \mapsto b_{4,0}, d \mapsto a_{5,0},
e \mapsto a_{9,1}, f \mapsto b_{10,0}, g \mapsto a_{11,1}, h \mapsto -c_{12,1}.
$$
\end{remark}
\begin{landscape}
\begin{table}[p]
{\tiny
\xymatrix@=0pt{
&&&&&&&&&&&&&&&&&&&&&&\\
21& & &t^6w_{2,9} & & & & & & & & & & & & & & & & & & &\\
20& & & & & & & &t^{10}w_{7,0} & & & &t^7w_{4,6}w_{7,1} & & & &t^{10}w_{7,0}w_{8,0} & & & &t^7w_{4,6}w_{7,1}w_{8,0} & &\\
19& & & & & & & &ut^9w_{7,0} & & & &ut^6w_{4,6}w_{7,1} & & & &ut^9w_{7,0}w_{8,0} & & & &ut^6w_{4,6}w_{7,1}w_{8,0} & &\\
18& & & &t^6w_{3,6} &t^6w_{4,6} & & &t^{9}w_{7,1} &t^6w_{8,6} & & &t^6w_{4,6}w_{7,0} &t^6w_{4,6}w_{8,0} & & &t^9w_{7,1}w_{8,0} &t^6w_{8,6}w_{8,0} & & &t^6w_{4,6}w_{7,0}w_{8,0} &t^6w_{4,6}w^{2}_{8,0} &\\
17& & & & & & & & & & & & & & & & & & & & & &\\
16& & & & & & & & & & & & & & & & & & & & & &\\
15& & & & & & & & & & & & & & & & & & & & & &\\
14& & & &t^6w_{3,2} & & & &t^7w_{7,1} & & & &t^6w_{3,2}w_{8,0} & & & &t^7w_{7,1}w_{8,0} & & & &t^6w_{3,2}w^{2}_{8,0} & &\\
13& & &t^6w_{2,1} &t^6w_{3,1} & & & &ut^6w_{7,1} & & & &t^6w_{3,2}w_{8,0} & & & &ut^6w_{7,1}w_{8,0} & & & &ut^6w_{3,1}w^{2}_{8,0} & &\\
& & & & t^6w^{-}_{3,1} & & & & & & & & & & & & & & & & & &\\
12&t^6 & & & & & & &t^6w_{7,0} &t^6w_{8,0} & & &t^6w_{11,0} &t^6w_{12,0} & & & t^6w_{7,0}w_{8,0} & t^6w^2_{8,0} & & & t^6w_{11,0}w_{8,0}& t^6w_{12,0}w_{8,0} &\\
11& & & & & & & & & & & & & & & & & & & & & &\\
10& & & & & & & & & & & & & & & & & & & & & &\\
9& & &w_{2,9} & & & & & & & & & & & & & & & & & & &\\
8& & & & & & & &t^4w_{7,0} & & & &tw_{4,6}w_{7,1} & & & &t^4w_{7,0}w_{8,0} & & & &tw_{4,6}w_{7,1}w_{8,0} & &\\
7& & & & & & & &ut^3w_{7,0} & & & &uw_{4,6}w_{7,1} & & & &ut^3w_{7,0}w_{8,0} & & & &uw_{4,6}w_{7,1}w_{8,0} & &\\
6& & & &w_{3,6} &w_{4,6} & & &t^3w_{7,1} &w_{8,6} & & &w_{4,6}w_{7,0} &w_{4,6}w_{8,0} & & &t^3w_{7,1}w_{8,0} &w_{8,6}w_{8,0} & & &w_{4,6}w_{7,0}w_{8,0} &w_{4,6}w^{2}_{8,0} &\\
5& & & & & & & & & & & & & & & & & & & & & &\\
4& & & & & & & & & & & & & & & & & & & & & &\\
3& & & &  & & & & & & & & & & & & & & & & & &\\
2& & & &w_{3,2} &  & & &tw_{7,1} & & & &w_{3,2}w_{8,0} & & & &tw_{7,1}w_{8,0} & & & &w_{3,2}w^{2}_{8,0} & &\\
& && & &  & & & & & & & & & & & & & & & & &\\
1& & &w_{2,1} &w_{3,1} & & & &uw_{7,1} & & & &w_{3,1}w_{8,0} & & & &uw_{7,1}w_{8,0} & & & &w_{3,1}w^{2}_{8,0} & &\\
& & & &w^{-}_{3,1}  & & & & & & & & & & & & & & & & & &\\
0&1 & & & & & & &w_{7,0} &w_{8,0} & & &w_{11,0} &w_{12,0} & & &w_{7,0}w_{8,0} &w^{2}_{8,0} & & &w_{11,0}w_{8,0} &w_{12,0}w_{8,0} &\\
&0 &  & 2 & 3 & 4 & & &7 &8 & & &11 &12 & & &15 &16 && &19 &20 & 
\ar@{-}"1,2"+<-7pt,-2pt>;"26,2"+<-7pt,-5pt>;
\ar@{-}"1,2"+<7pt,-2pt>;"26,2"+<7pt,-5pt>;
\ar@{-}"1,3"+<5pt,-2pt>;"26,3"+<5pt,-5pt>;
\ar@{-}"1,4"+<15pt,-2pt>;"26,4"+<15pt,-5pt>;
\ar@{-}"1,5"+<16pt,-2pt>;"26,5"+<16pt,-5pt>;
\ar@{-}"1,6"+<14pt,-2pt>;"26,6"+<14pt,-5pt>;
\ar@{-}"1,7"+<4pt,-2pt>;"26,7"+<4pt,-5pt>;
\ar@{-}"1,8"+<5pt,-2pt>;"26,8"+<5pt,-5pt>;
\ar@{-}"1,9"+<17pt,-2pt>;"26,9"+<17pt,-5pt>;
\ar@{-}"1,10"+<14pt,-2pt>;"26,10"+<14pt,-5pt>;
\ar@{-}"1,11"+<4pt,-2pt>;"26,11"+<4pt,-5pt>;
\ar@{-}"1,12"+<5pt,-2pt>;"26,12"+<5pt,-5pt>;
\ar@{-}"1,13"+<22pt,-2pt>;"26,13"+<22pt,-5pt>;
\ar@{-}"1,14"+<21pt,-2pt>;"26,14"+<21pt,-5pt>;
\ar@{-}"1,15"+<4pt,-2pt>;"26,15"+<4pt,-5pt>;
\ar@{-}"1,16"+<5pt,-2pt>;"26,16"+<5pt,-5pt>;
\ar@{-}"1,17"+<25pt,-2pt>;"26,17"+<25pt,-5pt>;
\ar@{-}"1,18"+<21pt,-2pt>;"26,18"+<21pt,-5pt>;
\ar@{-}"1,19"+<4pt,-2pt>;"26,19"+<4pt,-5pt>;
\ar@{-}"1,20"+<5pt,-2pt>;"26,20"+<5pt,-5pt>;
\ar@{-}"1,21"+<30pt,-2pt>;"26,21"+<30pt,-5pt>;
\ar@{-}"1,22"+<23pt,-2pt>;"26,22"+<23pt,-5pt>;
\ar@{-}"1,23"+<5pt,-2pt>;"26,23"+<5pt,-5pt>;
\ar@{-}"1,2"+<-7pt,-2pt>;"1,23"+<5pt,-2pt>;
\ar@{-}"2,2"+<-7pt,-5pt>;"2,23"+<5pt,-5pt>;
\ar@{-}"3,2"+<-7pt,-5pt>;"3,23"+<5pt,-5pt>;
\ar@{-}"4,2"+<-7pt,-5pt>;"4,23"+<5pt,-5pt>;
\ar@{-}"5,2"+<-7pt,-5pt>;"5,23"+<5pt,-5pt>;
\ar@{-}"6,2"+<-7pt,-5pt>;"6,23"+<5pt,-5pt>;
\ar@{-}"7,2"+<-7pt,-5pt>;"7,23"+<5pt,-5pt>;
\ar@{-}"8,2"+<-7pt,-5pt>;"8,23"+<5pt,-5pt>;
\ar@{-}"9,2"+<-7pt,-5pt>;"9,23"+<5pt,-5pt>;
\ar@{-}"11,2"+<-7pt,-5pt>;"11,23"+<5pt,-5pt>;
\ar@{-}"12,2"+<-7pt,-5pt>;"12,23"+<5pt,-5pt>;
\ar@{-}"13,2"+<-7pt,-5pt>;"13,23"+<5pt,-5pt>;
\ar@{-}"14,2"+<-7pt,-5pt>;"14,23"+<5pt,-5pt>;
\ar@{-}"15,2"+<-7pt,-5pt>;"15,23"+<5pt,-5pt>;
\ar@{-}"16,2"+<-7pt,-5pt>;"16,23"+<5pt,-5pt>;
\ar@{-}"17,2"+<-7pt,-5pt>;"17,23"+<5pt,-5pt>;
\ar@{-}"18,2"+<-7pt,-5pt>;"18,23"+<5pt,-5pt>;
\ar@{-}"19,2"+<-7pt,-5pt>;"19,23"+<5pt,-5pt>;
\ar@{-}"20,2"+<-7pt,-5pt>;"20,23"+<5pt,-5pt>;
\ar@{-}"21,2"+<-7pt,-5pt>;"21,23"+<5pt,-5pt>;
\ar@{-}"23,2"+<-7pt,-5pt>;"23,23"+<5pt,-5pt>;
\ar@{-}"25,2"+<-7pt,-5pt>;"25,23"+<5pt,-5pt>;
\ar@{-}"26,2"+<-7pt,-5pt>;"26,23"+<5pt,-5pt>;
}
}
\caption{Linear generators in  $\{{E_6^{n,m}}^{C_8}\}_{n,m=0,\ldots,21}$ }
\label{table_E6_21x21}
\end{table}
\end{landscape}


\begin{thebibliography}{99}

\bibitem{AM2004} A.~Adem, R.J.~Milgram, {\it Cohomology of finite groups}, 
Second edition. Grundlehren der Mathematischen Wissenschaften, 309. Springer-Verlag, Berlin, 2004.

\bibitem{BS2008} D.J.~Benson, S.D.~Smith, {\it Classifying spaces of sporadic groups}, Mathematical Surveys and Monographs, 147. American Mathematical Society, Providence, RI, 2008.

\bibitem{BLO2003} C.~Broto, R.~Levi, B.~Oliver, {\it The homotopy theory of fusion systems}, J. Amer. Math. Soc. 16 (2003), no. 4, 779-856. 

\bibitem{CMM1999} J.F.~Carlson, J.S.~Maginnis, R.J.~Milgram, {\it
The cohomology of the sporadic groups $J_2$ and $J_3$}, J. Algebra 214 (1999), no. 1, 143-173. 


\bibitem{CE1956} H.~Cartan and S.~Eilenberg, {\it Homological Algebra}, Princenton University Press, 1956.

\bibitem{D2012} A.~D\'{i}az, {\it A spectral sequence for fusion systems}, Algebraic \& Geometric Topology (to appear),  arXiv:1109.1952v3.

\bibitem{E1991} L.~Evens, {\it The cohomology of groups}, Oxford Mathematical Monographs. Oxford Science Publications. The Clarendon Press, Oxford University Press, New York, 1991.

\bibitem{FF2009} R.J.~Flores, R.M.~Foote, {\it Strongly closed subgroups of finite groups}, Adv. Math. 222 (2009), no. 2, 453-484. 

\bibitem{KING} S.~King, {\it Modular Cohomology Ring of $J_2$ over $\FF_3$}, \url{http://users.minet.uni-jena.de/~king/cohomology/nonprimepower/J2mod3.html}

\bibitem{KINGGREEN} S.~King, D.~Green, {\it Mod-$p$ Group Cohomology Package for SAGE}, \url{http://sage.math.washington.edu/home/SimonKing/Cohomology/}.

\bibitem{KINGGREENELLIS} S.~King, D.~Green, G.~Ellis, {\it The mod-2 cohomology ring of the third Conway group is Cohen-Macaulay}, Algebr. Geom. Topol. 11 (2011), no. 2, 719-734. 

\bibitem{L1992thesis} I.~Leary, {\it The cohomology of certain groups}, Ph. D. Thesis, Cambridge 1990.

\bibitem{L1992} I.~Leary, {\it The mod-p cohomology rings of some p-groups}, Math. Proc. Cambridge Philos. Soc. 112 (1992), no. 1, 63-75. 

\bibitem{L1993} I.~Leary, {\it A differential in the Lyndon-Hochschild-Serre spectral sequence}, J. Pure Appl. Algebra 88 (1993), no. 1-3, 155-168. 

\bibitem{RV2004} A.~ Ruiz, A.~Viruel, {\it The classification of p-local finite groups over the extraspecial group of order $p^3$ and exponent $p$}, Math. Z. 248 (2004), no. 1, 45–65.

\bibitem{TY1996} M.~Tezuka, N.~Yagita, {\it On odd prime components of cohomologies of sporadic simple groups and the rings of universal stable elements}, J. Algebra 183 (1996), no. 2, 483-513
. 
\bibitem{Y1998} N.~Yagita, {\it On odd degree parts of cohomology of sporadic simple groups whose Sylow $p$-subgroup is the extra-special $p$-group of order $p^3$}, J. Algebra 201 (1998), no. 2, 373-391. 

\end{thebibliography}
\end{document}